\documentclass[11pt, amsfonts]{amsart}

\usepackage{amsmath,amssymb,amsthm,enumerate,comment}
\usepackage[all]{xy}
\usepackage{color}
\usepackage{graphicx}
\usepackage[T1]{fontenc}
\usepackage{mathrsfs}
\usepackage[export]{adjustbox}

\SelectTips{eu}{12}

\theoremstyle{definition}
\newtheorem{theorem}{Theorem}[section]
\newtheorem{definition}[theorem]{Definition}
\newtheorem{lemma}[theorem]{Lemma}
\newtheorem{proposition}[theorem]{Proposition}
\newtheorem{corollary}[theorem]{Corollary}

\newtheorem{notation}[theorem]{\rm Notation}

\makeatletter
\@addtoreset{equation}{section} 
\makeatother

\DeclareMathOperator{\Hom}{Hom}
\DeclareMathOperator{\rot}{rot}

\DeclareMathOperator{\PSL}{PSL}

\DeclareMathOperator{\Homeo}{Homeo}

\DeclareMathOperator{\Isom}{Isom}

\newcommand{\RR}{\mathbb{R}}

\newcommand{\ZZ}{\mathbb{Z}}

\newcommand{\HH}{\mathbb{H}}
\newcommand{\GG}{\Gamma}
\renewcommand{\gg}{\gamma}

\newcommand{\trho}{\widetilde{\rho}}
\newcommand{\tphi}{\widetilde{\phi}}

\newcommand{\Int}{\mathrm{Int}}
\newcommand{\hG}{\widehat{G}}
\newcommand{\trot}{\widetilde{\rot}}
\newcommand{\trhod}{\widetilde{\rho^{(d)}}}
\newcommand{\hphi}{\widehat{\phi}}
\newcommand{\hGG}{\widehat{\GG}}
\newcommand{\hgg}{\widehat{\gg}}

\newcommand{\tHomeo}{\mathrm{H}\widetilde{\mathrm{omeo}}}

\newcommand{\HHH}{\mathrm{H}}

\newcommand{\maru}{\color{blue}}

\makeatletter
\@namedef{subjclassname@2020}{%
  \textup{2020} Mathematics Subject Classification}
\makeatother

\title{Isolated circular orders on free products of cyclic groups}

\author{Chihaya Jibiki}
\address{Department of Mathematics, School of Science, Institute of Science Tokyo, 2-12-1 Ookayama, Meguro-ku, Tokyo 152-8550 Japan}
\email{chihaya.j@gmail.com}

\author{Shuhei Maruyama}
\address{School of Mathematics and Physics, College of Science and Engineering, Kanazawa University, Kakuma-machi, Kanazawa, Ishikawa, 920-1192, Japan}
\email{smaruyama@se.kanazawa-u.ac.jp}


\begin{document}

\begin{abstract}
  In this paper, we construct countably many isolated circular orders on the free products $G = F_{2n} \ast \mathbb{Z}_{m_1} \ast \cdots \ast \mathbb{Z}_{m_k}$ of cyclic groups.
  Moreover, we prove that these isolated circular orders are not the automorphic images of the others.
  By using these isolated circular orders, we also construct countably many isolated left orders on a certain central $\mathbb{Z}$-extension of $G$, which are not the automorphic images of the others.
\end{abstract}

\maketitle

\section{Introduction} \label{sec:intro}
For a group $\GG$, a (left invariant) circular order on $\GG$ is a cocycle $c \colon \GG^3 \to \{\pm 1, 0\}$ with certain conditions (see Section \ref{sec:prel}).
Let $CO(\GG)$ be the space of circular orders on $\GG$.
This space $CO(\GG)$ is a compact, metrizable and totally disconnected space if $\GG$ is countable.
Hence its isolated points, which are called \emph{isolated circular orders}, have been studied (\cite{MR3813208}, \cite{MR3887426}, \cite{MR4033501}, \cite{MR4055461}, \cite{MR4584769}).

Though circular orders are combinatorial objects, they can be studied in dynamical ways.
In fact, if $\GG$ is countable, there is a procedure to construct a faithful action
of $\GG$ on the circle from a circular order, which is called the \emph{dynamical realization}.
More strongly, it is known that a countable group $\GG$ admits a circular order if and only if $\GG$ acts on the circle faithfully.

The dynamical realization gives rise to a map
\[
  R \colon CO(\GG) \to \Hom(\GG, \Homeo_+(S^1))/ \sim,
\]
where $\sim$ denotes the equivalence relation by conjugacy.
Mann and Rivas proved in \cite{MR3887426} that this map $R$ is continuous.
Employing this and ping-pong arguments, they proved that the free groups of even rank admit infinitely many isolated circular orders which are not conjugate to the others.
Contrastingly, Malicet, Mann, Rivas and Triestino proved in \cite{MR4033501} that the
free groups of odd rank do not admit isolated circular orders.
In \cite{MR4055461}, Matsumoto also used the continuity of $R$ and ping-pong arguments to prove that
$\PSL(2, \ZZ) \cong \ZZ_2 \ast \ZZ_3$ admits infinitely many isolated circular orders which are not the automorphic images of the others.
Along this line, we prove the following.

\begin{theorem}\label{thm:main}
  Let $n \in \ZZ_{\geq 0}, k \in \ZZ_{\geq 1}$ and $m_1, \ldots, m_k \in \ZZ_{\geq 2}$ such that $(n,k,m_1, \cdots, m_k) \neq (0,1,m_1), (0,2,2,2)$.
  Then the free product $G = F_{2n} \ast \ZZ_{m_1} \ast \cdots \ast \ZZ_{m_k}$ admits countably many isolated circular orders which are not the automorphic images of the others.
\end{theorem}

The groups excluded in Theorem \ref{thm:main} are finite cyclic groups and the free product $\ZZ_2 \ast \ZZ_2$.
The finite cyclic groups admit only finitely many circular orders.
The group $\mathbb{Z}_2 \ast \mathbb{Z}_2$ admits just four circular orders \cite{MR3784820}.
In particular, these circular orders are all isolated.

A left order on a group is a total order which is invariant under the left-multiplication.
As is the case of $CO(\GG)$, the space $LO(\GG)$ of left orders on $\GG$ is compact, metrizable and totally disconnected if $\GG$ is countable.
Isolated left orders on several groups has been studied.
For example, the following groups admit no isolated left orders: free abelian groups of rank greater than one \cite{MR2069015}, free groups of rank greater than one \cite{MR2766228}, free products of two left-orderable groups \cite{MR2859890}, virtually solvable groups \cite{MR3460331}, and surface groups of genus greater than one and certain amalgamated products of free groups \cite{MR3664524}.
Isolated left orders on the following group have been constructed: the braid groups $B_n$ \cite{MR1859702, MR2463428}, the fundamental groups of the complement of torus knots \cite{MR2745552, MR2998793}, a certain family of groups including torus knot groups and some of their amalgamated products \cite{MR3200370}, amalgamated free products over $\ZZ$ of groups with isolated left orders \cite{MR3476136}, and direct products $F_{2n} \times \ZZ$ \cite{MR3887426}.

Let $\hGG$ be a central $\ZZ$-extension of $\GG$.
Mann and Rivas \cite{MR3887426} provided a criterion for a certain left order on $\hGG$ to be isolated in terms of the circular orders on $\GG$.
They used this to prove that the direct product $F_{2n} \times \ZZ$ admits countably many isolated left orders which are not conjugate to the others.
Similarly, Matsumoto \cite{MR4055461} proved that the braid group $B_3$ of $3$-strands, which is a central extension of $\PSL(2,\ZZ)$, admits countably many isolated left orders which are not the automorphic images of the others.

Applying the criterion by Mann--Rivas to Theorem \ref{thm:main}, we prove the following.

\begin{theorem}\label{thm:main_left}
  Let $n,k,m_1, \ldots, m_k$ be as in Theorem \ref{thm:main}.
  Let $\hG$ be the group given as
  \[
    \hG = \langle e_1, \ldots, e_k, h_1, \ldots, h_{2n}, z \mid e_1^{m_1} = \cdots = e_k^{m_k} = z, [h_1, z] = \cdots = [h_{2n}, z] = 1 \rangle.
  \]
  Then $\hG$ admits countably many isolated left orders which are not the automorphic images of the others.
\end{theorem}

\section{Preliminaries} \label{sec:prel}
\subsection{Circular orders}

In this section, we review some facts about circular orders. 
We refer the reader to \cite{MR2172491} and \cite{MR3813208} for further details.

\begin{definition}
  Let $\GG$ be a group.
  A map $c \colon \GG^3 \to \{0, \pm 1\}$ is called a \emph{circular order on $\GG$} if
  \begin{enumerate}
    \item $c(g_1 , g_2 , g_3) = 0$ if $g_i = g_j$ for some $i\neq j$,
    \item $c(g_2, g_3, g_4)-c(g_1, g_3, g_4)+c(g_1, g_2, g_4)-c(g_1, g_2, g_3) = 0$ for every $g_1, g_2, g_3, g_4 \in \GG$, and
    \item $c(gg_1, gg_2, gg_3) = c(g_1, g_2, g_3)$ for every $g, g_1, g_2, g_3 \in \GG$.
  \end{enumerate}
\end{definition}
For an automorphism $\phi \colon \GG \to \GG$, we define $c_{\phi} \colon \GG^3 \to \{ 0, \pm 1 \}$ by
\[
  c_{\phi}(g_1, g_2, g_3) = c(\phi(g_1), \phi(g_2), \phi(g_3))
\]
for $g_1, g_2, g_3 \in \GG$.
Then $c_{\phi}$ is a circular order on $\GG$, called an \emph{automorphic image} of $c$.

Let $CO(\GG)$ be the set of circular orders on $\GG$.
The product topology of $\{0, \pm 1\}^\GG$ induces the topology of $CO(\GG)$.
Then the space $CO(\GG)$ is known to be compact and totally disconnected, and it is also metrizable if $\GG$ is countable \cite{MR2069015}.

\begin{definition}
  A circular order $c$ is said to be \emph{isolated} if $c$ is an isolated point in $CO(\GG)$.
\end{definition}

For a circular order $c$ and a finite set $S$ of $\GG$, we set $U_{c,S} = \{ c' \in CO(\GG) \colon c'|_{S^3} = c|_{S^3} \}$.
Then the neighborhood basis of $c$ is given by $\{ U_{c,S} \colon S \text{ is a finite subset of } \GG \}$.
Hence a circular order $c$ is isolated if and only if there exists a finite subset $S$ of $\GG$ such that $U_{c,S} = \{ c \}$.

Every circular order on a countable group $\GG$ is realized as an orbit of some action of $\GG$ on the circle.
Indeed, let $c$ be a circular order on $\GG$ and $x_0$ a point in $S^1$.
Here we regard $S^1$ as the quotient $\mathbb{R}/\mathbb{Z}$.
We fix an enumeration $\{ g_i \}_{i \geq 0}$ with $g_0 = 1$ and define an embedding $\iota \colon \GG \to S^1$ as follows.
First we set $\iota(g_0) = x_0$ and $\iota(g_1) = x_0 + 1/2$.
If $\iota(g_i)$ is defined for all $i \leq n$, we define $\iota(g_{n+1})$ as the midpoint of the unique connected component of $S^1 \setminus \{ \iota(g_0), \cdots, \iota(g_n) \}$ such that
\begin{align*}
    c(g_i, g_j, g_k) = \mathrm{ord}(\iota(g_i), \iota(g_j), \iota(g_k)).
\end{align*}
Here the map $\mathrm{ord} \colon (S^1)^3 \to \{ 0, \pm 1 \}$ is defined by
\begin{align*}
    \mathrm{ord}(x,y,z) =
    \begin{cases}
        1 & \text{ if $x,y,z$ are distinct and arranged counterclockwise}, \\
        -1 & \text{ if $x,y,z$ are distinct and arranged clockwise}, \\
        0 & \text{ otherwise.}
    \end{cases}
\end{align*}
Then the self left-action of $\GG$ induces a continuous order-preserving action on $\iota(\GG) \subset S^1$.
This action extends to a continuous action on the closure of $\iota(\GG)$, and also extends to a continuous action on $S^1$ by setting the action on the complements of the closure of $\iota(\GG)$ to be affine.
Now we obtain an action $\rho_{c} \colon \GG \to \mathrm{Homeo}_+(S^1)$, which is called the \emph{dynamical realization of $c$ based at $x_0$}.
By construction, the orbit $\rho_c(\GG)\cdot x_0$ recovers the circular order $c$.

It is known that the conjugacy class of the dynamical realization $\rho_{c}$ does not depend on the choices of the enumeration $\{ g_i \}_{i\geq 0}$ or the basepoint $x_0$.
Hence the dynamical realization gives rise to a map
\begin{align*}
    R \colon CO(\GG) \to \mathrm{Hom}(\GG, \mathrm{Homeo}_+(S^1))/\sim,
\end{align*}
where $\mathrm{Hom}(\GG, \mathrm{Homeo}_+(S^1))/\sim$ denotes the quotient of $\mathrm{Hom}(\GG, \mathrm{Homeo}_+(S^1))$ by the equivalence relation of conjugacy.
This map $R$ is called the \emph{realization map}.
Note that the set $\mathrm{Hom}(\GG, \mathrm{Homeo}_+(S^1))/\sim$ is equipped with the quotient topology of the pointwise convergence topology on $\mathrm{Hom}(\GG, \mathrm{Homeo}_+(S^1))$.

We will use the following properties of the realization map $R$.

\begin{theorem}[{\cite[Proposition 3.3]{MR3887426}}]\label{thm:continuous}
    The realization map $R$ is continuous.
\end{theorem}

\begin{theorem}[{\cite[Lemma 3.12]{MR3887426}}]\label{thm:dyn_real_cont}
  Let $c \in CO(\GG)$ and let $\rho \colon \GG \to \Homeo_+(S^1)$ be a dynamical realization of $c$ based at $x_0$.
  Let $U$ be any neighborhood of $\rho$ in $\Hom(\GG, \Homeo_+(S^1))$.
  Then, there exists a neighborhood $V$ of $c$ in $CO(\GG)$ such that each order in $V$ has a dynamical realization based at $x_0$ contained in $U$.
\end{theorem}

\begin{theorem}[{\cite[Corollary 3.24]{MR3887426}}]\label{thm:dyn_real_criterion}
  Suppose that $\rho \colon \GG \to \Homeo_+(S^1)$ is an action with an exceptional minimal set $K$, that $\rho(\GG)$ acts transitively on the set of connected components of $S^1 \setminus K$, and that, for some component $I$, the stabilizer of $I$ is nontrivial and acts on $I$ as the dynamical realization of a linear order with basepoint $x_0 \in I$. Then $\rho$ is the dynamical realization of a circular order with basepoint $x_0$.
\end{theorem}

\subsection{Left orders}
In this subsection, we recall the definition and properties of left orders.
We refer the reader to \cite{MR3560661} and \cite{deroin2014groups} for further details.
\begin{definition}
  Let $\GG$ be a group.
  A total order $<$ on $\GG$ is called a \emph{left order on} $\GG$ if for every $g \in \GG$ and $g_1, g_2 \in \GG$ satisfying $g_1 < g_2$, the inequality $gg_1 < gg_2$ holds.
\end{definition}
For a left order $<$ on $\GG$ and an automorphism $\phi$ of $\GG$, we define another left order $<_{\phi}$ on $\GG$ as follows: $g_1 <_{\phi} g_2$ if and only if $\phi(g_1) < \phi(g_2)$.

Let $LO(\GG)$ be the set of left orders on $\GG$.
For a left order $<$ on $\GG$, define $\lambda \colon \GG\setminus \{ 1 \} \to \{ \pm 1 \}$ as follows; $\lambda(g) = 1$ if and only if $1 < g$.
Hence $LO(\GG)$ can be seen as a subspace of $\{ \pm 1 \}^{\GG \setminus \{ 1 \}}$.

Let $c$ be a circular order on $\GG$ is called a \emph{coboundary} if there exists a left invariant function $c' \colon \GG^2 \to \{ \pm 1, 0 \}$ such that $c(f,g,h) = c'(g,h) - c'(f,h) + c'(f,g)$ for every $f,g,h \in \GG$.
A coboundary $c$ induces a left order on $\GG$ by $g\leq h$ if and only if $c'(g,h)\geq 0$.

\begin{definition}[{\cite[Definition 3.14]{MR3887426}}]
    Let $\GG$ be a group and $c$ a circular order on $\GG$.
    A subgroup $H$ of $\GG$ is said to be \emph{convex} if the restriction of $c$ to $H$ is a left order (that is, a coboundary), and if whenever one has $c(h_1, g, h_2) = +1$ and $c(h_1, 1, h_2) = +1$ with $h_1, h_2$ in $H$ and $g \in \GG$, then $g \in H$.
\end{definition}

As shown in \cite[Lemma 3.15]{MR3887426}, there is a unique maximal convex subgroup for any $c\in CO(\GG)$, which we call the \emph{linear part} of $c$.

The linear part is characterized in terms of dynamics as follows.
\begin{theorem}[{\cite[Part of Proposition 3.17]{MR3887426}}]\label{thm:linear_part_dyn_description}
  Let $\rho$ be the dynamical realization of a circular order $c$ based at $x_0$.
  If $\rho$ admits an exceptional minimal set $K$, then the linear part is the stabilizer of the connected component of $S^1\setminus K$ that contains $x_0$.
\end{theorem}

\subsection{Known results on isolated circular orders on free products of cyclic groups}
Let $n$ be an integer grater than $1$.
Since $\ZZ_n$ is finite, $CO(\ZZ_n)$ is a finite set.
Clay, Mann and Rivas \cite{MR3784820} proved that $\mathbb{Z}_2* \mathbb{Z}_2$ admits just four circular orders.
Hence, circular orders on the groups above are all isolated.

Mann and Rivas \cite{MR3887426} proved that for every $n \geq 1$, the non-abelian free group $F_{2n}$ of rank $2n$ admits countably many isolated circular orders which are not conjugate to each other.
After that, Malicet, Mann, Rivas and Triestino \cite{MR4033501} proved that the non-abelian free group $F_{2n-1}$ of rank $2n-1$ admits no isolated circular orders for every $n \geq 1$.
Matsumoto \cite{MR4055461} proved that $\PSL(2,\ZZ) \cong \ZZ_2 \ast \ZZ_3$ admits countably many isolated circular orders which are not the automorphic images of the others.
Applying the argument of Matsumoto \cite[Proof of Theorem 9.2]{MR4055461} to the isolated circular orders on $F_{2n}$ obtained by Mann and Rivas \cite{MR3887426}, we can show that these isolated circular orders are not the automorphic images of the others as well.

By lifting these isolated orders to certain central extensions, they obtain isolated left orders.
In fact, Mann and Rivas \cite{MR3887426} proved that the direct product $F_{2n} \times \ZZ$ admits countably many isolated left orders, and Matsumoto \cite{MR4055461} proved that the braid group $B_3$ of $3$-strand admits countably many isolated left orders, which are not the automorphic images of the others.


\section{Isolated circular orders}\label{sec:3}

\subsection{Construction of circular orders}
We fix $n \in \ZZ_{\geq 0}, k \in \ZZ_{\geq 1}$ and $m_1, \ldots, m_k \in \ZZ_{\geq 2}$ and set
\[
  G = \langle e_1, \cdots, e_k, h_1, \cdots h_{2n} \mid e_1^{m_1} = \cdots = e_k^{m_k} = 1 \rangle.
\]
That is, the group $G$ is the free product $F_{2n} * \ZZ_{m_1} * \cdots * \ZZ_{m_k}$ of cyclic groups.
Assume that $G$ is neither cyclic nor isomorphic to $\ZZ_2 * \ZZ_2$.

First we embed the group $G$ into $\Isom_+(\HH^2) \cong \PSL(2,\RR)$ as follows.
Take elliptic elements $\rho(e_1), \ldots, \rho(e_k)$ and hyperbolic elements $\rho(h_1), \ldots, \rho(h_{2n})$ of $\Isom_+(\HH^2)$ as in Figure \ref{fig:embedding}.
\begin{figure}[b]
  \centering
  \includegraphics{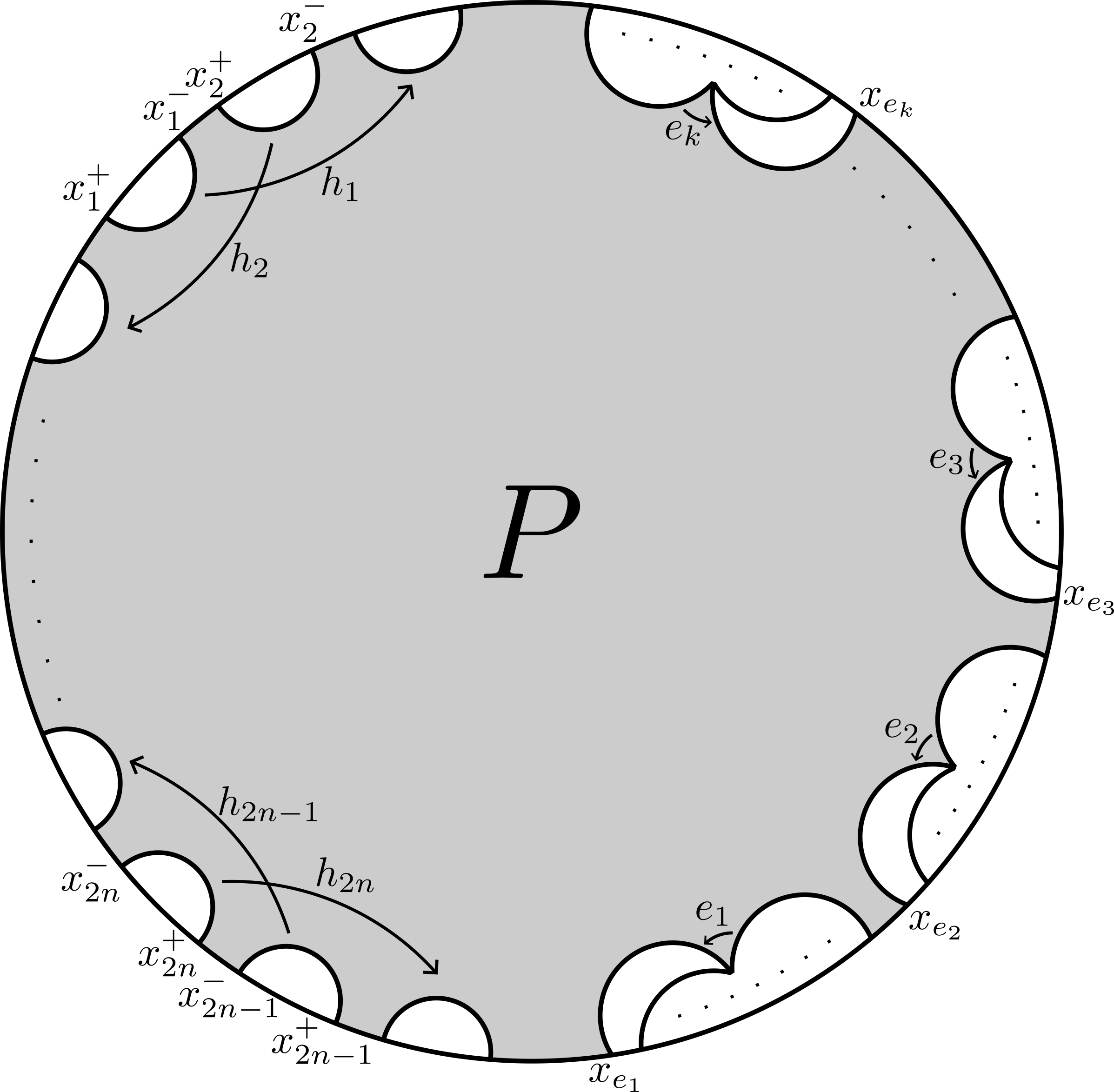}
  \caption{}\label{fig:embedding}
\end{figure}
Then, $\rho(e_i)$'s and $\rho(h_j)$'s induce a well-defined homomorphism $\rho \colon G \to \Isom_+(\HH^2)$.
By the standard ping-pong argument, we have the following.
\begin{lemma}\label{lem:inclusion}
  The homomorphism $\rho \colon G \to \Isom_+(\HH^2)$ is injective.
\end{lemma}

As $\Isom_+(\HH^2)$ is a subgroup of $\Homeo_+(S^1)$ through the action on the ideal boundary, the inclusion $G \to \Isom_+(\HH^2)$ given in Lemma \ref{lem:inclusion} gives rise to a homomorphism
\[
  \rho \colon G \to \Homeo_+(S^1).
\]
Under $\rho$, we regard $G$ as a subgroup of $\Homeo_+(S^1)$.
In particular, we abbreviate $\rho(g)(x)$ as $g(x)$ for $g \in G$ and $x \in S^1$.

Let $x_{e_1}$ be the point on $\partial \HH^2$ in Figure \ref{fig:embedding}.
Note that the limit set $L_{\rho}$ of $\rho$ is an exceptional minimal set since $G$ is non-amenable. 
Since the action $\rho$ is not minimal and $x_{e_1}$ is not contained in $L_{\rho}$, there exists a connected component $I_{o}$ of $S^1 \setminus L_{\rho}$ containing $x_{e_1}$.

\begin{notation}
  For distinct points $a,b \in S^1$, let $[a,b] \subset S^1$ be the closed interval from $a$ to $b$ according to the counterclockwise orientation.
  For a closed interval $J = [a,b]$ in $S^1$, we set $\partial_-(J) = a$ and $\partial_+(J) = b$.
\end{notation}

We set
\begin{align*}
  &J_i =
  \begin{cases}
    [e_i^{-1}(x_{e_i}), x_{e_{i+1}}] & i = 1, \ldots, k-1 \\
    [e_k^{-1}(x_{e_k}), h_1(x_1^+)] & i = k,
  \end{cases} \\
  &K_i^+ =
  \begin{cases}
    [x_i^+, h_{i+1}(x_{i+1}^+)] & i \colon \text{odd} \\
    [x_i^+, x_{i-1}^-] & i \colon \text{even, and}
  \end{cases}\\
  &K_i^- =
  \begin{cases}
    [h_i(x_i^-), x_{i+1}^-] & i \colon \text{odd} \\
    [h_i(x_i^-), h_{i+1}(x_{i+1}^+)] & i = 2, 4, \ldots, 2n-2 \\
    [h_{2n}(x_{2n}^-), x_{e_1}] & i = 2n.
  \end{cases}
\end{align*}
(See Figure \ref{fig:transition}.)
Let $P$ be the polygon depicted in Figure \ref{fig:embedding}, that is, the union of the intervals $J_i$'s, $K_i^{\pm}$'s of $\partial \HH^2$ and a fundamental domain of $\rho(G)$ in $\HH^2$.
Here we regard $P$ as a closed subset of $\HH^2 \cup \partial \HH^2$.

\begin{lemma}\label{lem:transitive_P}
  For every connected component $J$ of $P \cap \partial \HH^2$, there exists $g \in G$ such that $J$ is contained in $g(I_o)$.
\end{lemma}

\begin{proof}
  First note that the set of connected components of $P \cap \partial \HH^2$ is $\{ J_1, \ldots, J_k, K_1^{\pm}, \ldots, K_{2n}^{\pm} \}$.
  We set $J_0 = K_{2n}^{-}$ and $K_0 = J_k$.
  Then for every $i \in \{ 1, \ldots, k \}$, we have $\partial_+(J_{i-1}) = \partial_-(e_i(J_{i}))$.
  Moreover, we have $\partial_+(K_j^{\pm}) = \partial_-(h_{j+1}^{\pm 1}(K_{j+1}^{\pm}))$ for every odd integer $j$ in $\{ 1, \ldots, 2n \}$,
  and $\partial_+(K_j^{+}) = \partial_-(h_{j-1}^{-1}(K_{j-1}^{-}))$ and $\partial_+(K_{j-1}^{-}) = \partial_-(h_{j}(K_{j}^+))$ for every even integer $j$ in $\{ 1, \ldots, 2n \}$ (see Figure \ref{fig:transition}).
  Hence the gap $I_o$ is depicted as in Figure \ref{fig:interval}, where we set $\beta = e_1 \cdots e_k$ and $\gg_i = [h_{2i-1}, h_{2i}]$ for $i = 1, \ldots, n$.
  \begin{figure}
    \centering
    \includegraphics{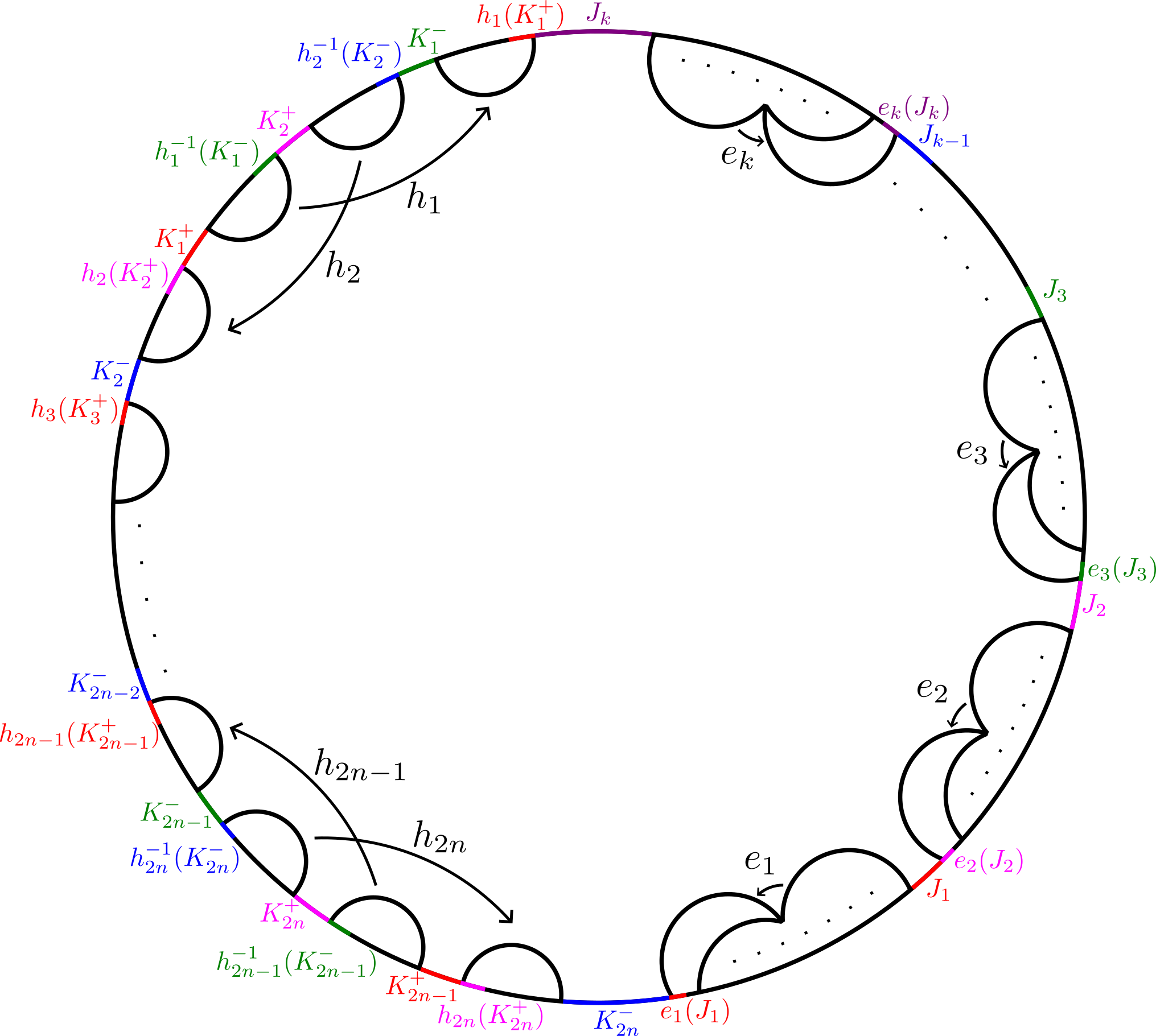}
    \caption{}\label{fig:transition}
  \end{figure}
  \begin{figure}
    \centering
    \includegraphics[scale=1.4,left]{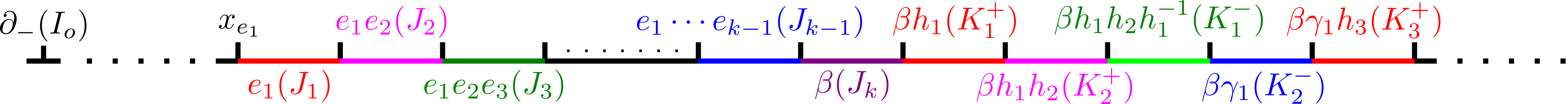}

    \bigskip
    \includegraphics[scale=1.4,right]{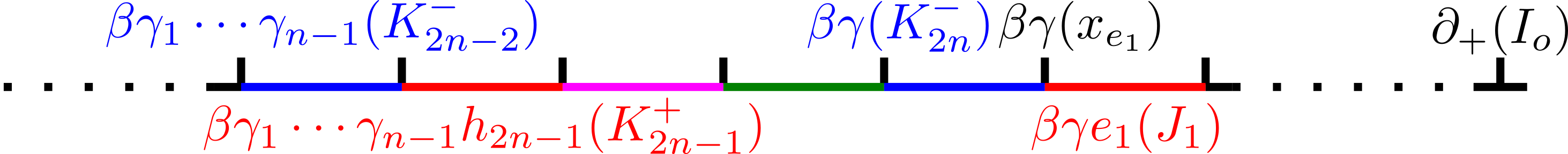}   
    \caption{}\label{fig:interval}
    \end{figure}
    This completes the proof.
\end{proof}

\begin{lemma}\label{lem:transitive_I_o}
  For every connected component $J$ of $\partial \HH^2 \setminus L_{\rho}$, there exists $g \in G$ such that $I_o = g(J)$.
  In particular, $\rho(G)$ acts transitively on the set of connected components of $\partial \HH^2 \setminus L_{\rho}$.
\end{lemma}

\begin{proof}
  Let $J$ be a connected component of $\partial \HH^2 \setminus L_{\rho}$.
  Since $P$ tessellates $\overline{\HH^2} \setminus L_{\rho}$, there exists $h \in G$ and $I \in \{ J_1, \ldots, J_k, K_1^{\pm}, \ldots, K_{2n}^{\pm} \}$ such that $I \subset h(J)$.
  In particular, together with Lemma \ref{lem:transitive_P}, there exists $g \in G$ such that $I_o \cap g(J) \neq \emptyset$.
  Since $g(J)$ is also a connected component of $\partial \HH^2 \setminus L_{\rho}$ and $I_o \cap g(J) \neq \emptyset$, either $I_o \cap g(J) \subset \partial I_o$ or $I_o = g(J)$ holds.
  However, the former statement contradicts to the fact that $L_{\rho}$ is homeomorphic to the Cantor set.
  Hence we have $I_o = g(J)$.
\end{proof}

We set $\alpha = e_1\cdots e_k[h_1, h_2]\cdots [h_{2n-1},h_{2n}]$.
Figure \ref{fig:interval} implies the following.
\begin{lemma}\label{lem:stabilizer}
  The stabilizer of $I_o$ is the infinite cyclic subgroup of $G$ generated by $\alpha$, and acts on $I_o$ as the dynamical realization of a linear order with basepoint $x_{e_1} \in I_o$.
\end{lemma}

Lemmas \ref{lem:transitive_I_o} and \ref{lem:stabilizer}, together with Theorems \ref{thm:dyn_real_criterion} and \ref{thm:linear_part_dyn_description}, imply the following.
\begin{proposition}\label{prop:co_from_rho}
  The action $\rho$ is a dynamical realization of a circular order $c \in CO(G)$ with basepoint $x_{e_1}$, and the linear part is the subgroup of $G$ generated by $\alpha$.
\end{proposition}

\subsection{Free subgroup of $G$}
In this section, we find a free subgroup $F$ of $G$ of finite index, the restriction of $\rho$ to which is a ping-pong action.

For $l \in \ZZ_{> 0}$, we set $[l] = \{ 1, \ldots, l \}$.
For $2 \leq t \leq k$, we set
\begin{align*}
  \Lambda_t &= \big(([m_1] \times \cdots \times [m_{t-1}])\setminus \{ (m_1,\ldots, m_{t-1}) \} \big) \times [m_t-1], \\
  \Xi_t &= [m_{t+1}] \times \cdots \times [m_k],
\end{align*}
where we regard $\Xi_k = \{ 0 \}$.
For $\lambda = (u_1, \ldots, u_t) \in \Lambda_t$ and $\xi = (u_{t+1}, \ldots, u_{k}) \in \Xi_t$, we set
\begin{align*}
  f_{\lambda} &= [e_t^{u_t}, e_{t-1}^{u_{t-1}}\cdots e_1^{u_1}], \\
  g_{\xi} &= e_k^{u_k}\cdots e_{t+1}^{u_{t+1}},
\end{align*}
where we regard $g_{\xi} = 1$ when $\xi \in \Xi_k$.
Let $S_1$ and $S_2$ be the subsets of $G$ defined by
\begin{align*}
  S_1 &= \left\{ {}^{g_{\xi}} f_{\lambda} \mid 2 \leq t \leq k, \lambda \in \Lambda_t, \xi \in \Xi_t \right\}, \\
  S_2 &= \left\{ {}^{e_k^{i_k}\cdots e_1^{i_1}} h_l \, \middle| \, l \in [2n], (i_1, \cdots i_k) \in [m_1] \times \cdots \times [m_k] \right\}.
\end{align*}
Let $F$ be the subgroup of $G$ generated by $S_1 \cup S_2$.
Let $G_1$ be the subgroup of $G$ generated by $\{ e_1, \ldots, e_k \}$.

\begin{notation}
  Let $\GG$ be a group.
  For elements $g,h$ of $\GG$, we set $[g,h] = ghg^{-1}h^{-1}$ and ${}^{g}{h} = ghg^{-1}$.
\end{notation}

\begin{lemma}\label{lem:sort_1}
  Let $w$ be an element of $G_1$.
  Then there exist an element $\gamma$ of the subgroup $\langle S_1 \rangle$ generated by $S_1$ and $i_1, \ldots, i_k$ such that $\gamma w = e_k^{i_k} \cdots e_1^{i_1}$.
\end{lemma}

\begin{proof}
  Let $\| w \|$ denote the word length of $w$ with respect to $\{ e_1, \ldots, e_k \}$.
  The statement is trivial if $\| w \| = 0$.
  Let $l \geq 1$ and assume that the statement holds for elements of $G_1$ whose word lengths are equal to $l-1$.
  Let $w$ be an element of $G_1$ with $\| w \| = l$.
  Let $e_t^{\pm 1}$ be the suffix of $w$ in the reduced word expression and set $w = ve_t^{\pm 1}$.
  Then $\| v \| = l-1$.
  By the above assumption, there exist $\gamma \in \langle S_1 \rangle$ and $i_1, \ldots, i_k$ such that $\gamma v = e_k^{i_k} \cdots e_1^{i_1}$.
  Hence we have
  \[
    \gamma w = e_k^{i_k} \cdots e_1^{i_1} \cdot e_t^{\pm 1}.
  \]

  If $t = 1$, the statement holds for $w$.
  Assume that $t \neq 1$.
  We set $\gamma' = {}^{e_k^{i_k} \cdots e_l^{i_l}} [e_l^{\pm 1}, e_{l-1}^{i_{l-1}}\cdots e_1^{i_1}]$.
  By the equality ${}^{cd}[a,b] = {}^{c}[da,b]\cdot {}^{c}[b,d]$, we have
  \begin{align*}
    \gamma' &= {}^{e_k^{i_k} \cdots e_{l+1}^{i_{l+1}} \cdot e_l^{i_l}} [e_l^{\pm 1}, e_{l-1}^{i_{l-1}}\cdots e_1^{i_1}] \\
    &= {}^{e_k^{i_k} \cdots e_{l+1}^{i_{l+1}}} [e_{l}^{i_l \pm 1}, e_{l-1}^{i_{l-1}}\cdots e_1^{i_1}] \cdot {}^{e_k^{i_k} \cdots e_{l+1}^{i_{l+1}}}[e_{l-1}^{i_{l-1}}\cdots e_1^{i_1}, e_l^{i_l}] \\
    &= {}^{e_k^{i_k} \cdots e_{l+1}^{i_{l+1}}} [e_{l}^{i_l \pm 1}, e_{l-1}^{i_{l-1}}\cdots e_1^{i_1}] \cdot \left({}^{e_k^{i_k} \cdots e_{l+1}^{i_{l+1}}}[e_l^{i_l}, e_{l-1}^{i_{l-1}}\cdots e_1^{i_1}]\right)^{-1}.
  \end{align*}
  Hence $\gamma'$ is an element of $\langle S_1 \rangle$.
  Moreover, we have
  \begin{align*}
    \gamma' \gamma w &= {}^{e_k^{i_k} \cdots e_l^{i_l}} [e_l^{\pm 1}, e_{l-1}^{i_{l-1}}\cdots e_1^{i_1}] \cdot e_k^{i_k} \cdots e_1^{i_1} \cdot e_t^{\pm 1} \\
    &= e_k^{i_k} \cdots e_{l+1}^{i_{l+1}} \cdot e_l^{i_l \pm 1} \cdot e_{l-1}^{i_{l-1}} \cdots e_1^{i_1}.
  \end{align*}
  This completes the proof.
\end{proof}

\begin{lemma}\label{lem:sort_2}
  For an element $g$ of $G$, there exist an element $f$ of $F$ and $i_1, \ldots, i_k$ such that $fg = e_k^{i_k}\cdots e_1^{i_1}$.
\end{lemma}

\begin{proof}
  If $g$ is contained in $G_1$, then the statement holds by Lemma \ref{lem:sort_1}.
  Assume that $g$ is not contained in $G_1$.
  Then there exists an element of $\{ h_1^{\pm 1}, \ldots, h_{2n}^{\pm 1} \}$ in the reduced word expression of $g$.
  Then there exist $w \in G_1$, $i \in \{ 1, \ldots, 2n \}$ and $g' \in G$ whose prefix in the reduced word expression is not equal to $h_i^{\mp 1}$ such that
  \[
    g = w \cdot h_i^{\pm 1} \cdot g'.
  \]
  Then, by Lemma \ref{lem:sort_1}, there exists $\gamma_1 \in \langle S_1 \rangle$ and $i_1, \ldots, i_k$ such that
  \[
    \gamma_1 g = e_{k}^{i_k} \cdots e_1^{i_1} \cdot h_i^{\pm 1} \cdot g'.
  \]
  We set $\gamma_2 = {}^{e_{k}^{i_k} \cdots e_1^{i_1}} h_i^{\mp 1}$, which is an element of $\langle S_2 \rangle$.
  Then we have
  \[
    \gamma_2 \gamma_1 g = g'.
  \]
  Note that $\gamma_2 \gamma_1$ is an element of $F$.
  This procedure strictly decreases the number of elements of $\{ h_1^{\pm 1}, \ldots, h_{2n}^{\pm 1} \}$ in the reduced word expression of $g$.
  Taking this procedure inductively, we obtain an element $\gamma \in F$ such that $\gamma g \in G_1$.
  Applying Lemma \ref{lem:sort_1} to $\gamma g$, we obtain the lemma.
\end{proof}

We set
\[
  Q = \bigcup_{i_k = 0}^{m_k-1}\cdots \bigcup_{i_1 = 0}^{m_1-1} e_k^{i_k} \cdots e_1^{i_1}(P).
\]
Then $Q\cap \HH^2$ turns out to be a fundamental domain of $F$ by the following two lemmas.

\begin{lemma}
  For every $f \in F\setminus \{ 1 \}$, the interior $\Int(Q)$ is disjoint from $f(\Int(Q))$.
\end{lemma}

\begin{proof}
  Assume that $\Int(Q)\cap f(\Int(Q))$ is non-empty.
  Since $P\cap \HH^2$ is a fundamental domain of $\rho(G)$, this implies that there exist $i_1, \cdots, i_k, j_1, \cdots, j_k$ such that
  \[
    e_k^{i_k} \cdots e_1^{i_1}(P) = fe_k^{j_k}\cdots e_1^{i_1}(P),
  \]
  and hence $e_k^{i_k} \cdots e_1^{i_1} = fe_k^{j_k}\cdots e_1^{i_1}$.
  Let $q \colon G \to \HHH_1(G)$ be the abelianization of $G$.
  Note that the abelianization $\HHH_1(G)$ is isomorphic to $\ZZ^{2n} \times \ZZ_{m_1} \times \cdots \times \ZZ_{m_k}$.
  The elements of $S_1$ are contained in the kernel of $q$ and the image $q(S_2)$ is contained in the subspace $\ZZ^{2n} \times \{ 0 \} \times \cdots \times \{ 0 \} \subset \HHH_1(G)$.
  Hence $q(e_k^{i_k} \cdots e_1^{i_1}) = q(fe_k^{j_k}\cdots e_1^{i_1})$ implies that $i_l = j_l$ modulo $m_l$ for every $1 \leq l \leq k$.
  Therefore we obtain $f = 1$.
\end{proof}

\begin{lemma}
  The equality $\displaystyle \bigcup_{f \in F} f(Q \cap \HH^2) = \HH^2$ holds.
\end{lemma}

\begin{proof}
  Since $P\cap \HH^2$ is a fundamental domain of $\rho(G)$, it is enough to prove that $g(P) \subset \bigcup_{f \in F} f(Q)$ for every $g \in G$.
  By Lemma \ref{lem:sort_2}, we take an element $f$ of $F$ such that $fg$ is of the form $e_k^{i_k}\cdots e_1^{i_1}$.
  Then we have $fg(P) = e_k^{i_k} \cdots e_1^{i_1}(P) \subset Q$, which implies $g(P) \subset \bigcup_{f \in F} f(Q)$.
\end{proof}

According to \cite{MR3887426}, we define the notion of ping-pong dynamics as follows.
\begin{definition}[{\cite[Definition 4.1]{MR3887426}}]
  Let $a_1, a_2, \ldots, a_n \in \Homeo_+(S^1)$.
  We say these elements have \emph{ping-pong dynamics} if there exist pairwise disjoint closed sets $D(a_i)$ and $D(a_i^{-1})$ such that, for each $i$, we have
  \[
    a_i (S^1\setminus D(a_i^{-1})) \subset D(a_i).
  \]
  We call such sets $D(a_i)$ and $D(a_i^{-1})$ satisfying $a_i (S^1 \setminus D(a_i^{-1})) \subset D(a_i)$ \emph{attracting domains} for $a_i$ and $a_i^{-1}$, respectively, and use the notation $D(a_i^{\pm 1})$ to denote $D(a_i) \cup D(a_i^{-1})$.
  These attracting domains need not be connected.
  In this case, we use the notation $D_1(s), D_2(s), \ldots$ for the connected components of $D(s)$. Note that the definition of ping-pong dynamics implies that for each $s \in \{ a_1^{\pm 1}, \ldots , a_n^{\pm 1}\}$, and for each domain $D_k(t)$ with $t \neq s^{-1}$, there exists a unique $j$ such that $s(D_k(t))$ lies in the interior of $D_j(s)$.
\end{definition}

In the following, we will prove that the elements of $S_1 \cup S_2$ have ping-pong dynamics.
For $\lambda = (u_1, \ldots, u_t) \in \Lambda_t$, we set
\begin{align}\label{align:J_lambda}
  J_{\lambda}^{-} &= [e_{t-1}^{u_{t-1}} \cdots e_1^{u_1}e_t^{u_t-1}(x_{e_t}), e_{t-1}^{u_{t-1}} \cdots e_1^{u_1}e_t^{u_t}(x_{e_t})], \nonumber \\
  J_{\lambda}^+ &= [e_t^{u_t}e_{t-1}^{u_{t-1}} \cdots e_1^{u_1}(x_{e_t}), e_t^{u_t}e_{t-1}^{u_{t-1}} \cdots e_1^{u_1}e_t^{-1}(x_{e_t})].
\end{align}
For every $l \in [2n]$, we set
\[
  J_{h_l}^- = [x_l^-, x_l^+], \  J_{h_l}^+ = [h_l(x_l^+), h_l(x_l^-)].
\]
The relationship between these intervals is as follows.

\begin{lemma}\label{lem:intersections}
  Let $2 \leq t, t' \leq k, \lambda \in \Lambda_t, \xi \in \Xi_t, \lambda' \in \Lambda_{t'}$ and $\xi' \in \Xi_{t'}$.
  Let $l, l' \in [2n]$ and $(i_1, \cdots i_k), (i'_1, \cdots, i'_k) \in [m_1]\times \cdots \times [m_k]$.
  Then the following hold;
  \begin{enumerate}
    \item $g_{\xi}(J_{\lambda}^+) \cap g_{\xi'}(J_{\lambda'}^{+}) \neq \emptyset$ if and only if $\lambda = \lambda'$ and $\xi = \xi'$,
    \item $\Int(g_{\xi}(J_{\lambda}^-)) \cap \Int(g_{\xi'}(J_{\lambda'}^-)) \neq \emptyset$ if and only if $\lambda = \lambda'$ and $\xi = \xi'$, 
    \item $g_{\xi}(J_{\lambda}^+) \cap g_{\xi'}(J_{\lambda'}^{-}) = \emptyset$,
    \item $g_{\xi}(J_{\lambda}^{+}) \cap e_k^{i_k} \cdots e_1^{i_1}(J_{h_l}^{\pm}) = \emptyset$,
    \item $g_{\xi}(J_{\lambda}^{-}) \cap e_k^{i_k} \cdots e_1^{i_1}(J_{h_l}^{\pm}) = \emptyset$,
    \item $e_k^{i_k} \cdots e_1^{i_1}(J_{h_l}^{+}) \cap e_k^{i'_k} \cdots e_1^{i'_1}(J_{h_{l'}}^{-}) = \emptyset$, and
    \item $e_k^{i_k} \cdots e_1^{i_1}(J_{h_l}^{\pm}) \cap e_k^{i'_k} \cdots e_1^{i'_1}(J_{h_{l'}}^{\pm}) = \emptyset$ if and only if $(i_1, \ldots, i_k) = (i'_1, \ldots, i'_k)$ and $l = l'$.
  \end{enumerate}
\end{lemma}

\begin{proof}
  We set $(\lambda, \xi) = (u_1, \ldots, u_t, u_{t+1}, \ldots, u_k) \in [m_1] \times \cdots \times [m_k]$.
  Then we have
  \begin{align*}
    g_{\xi}(J_{\lambda}^+) &= e_k^{u_k} \cdots e_1^{u_1}\big([x_{e_t}, e_t^{-1}(x_{e_t})]\big), \\
    g_{\xi}(J_{\lambda}^-) &= e_k^{u_k} \cdots e_{t+1}^{u_{t+1}} e_{t-1}^{u_{t-1}} \cdots e_1^{u_1}\big([e_t^{u_t -1}(x_{e_t}), e_t^{u_t}(x_{e_t})]\big).
  \end{align*}
  We set $(\lambda', \xi') = (u'_1, \ldots, u'_{t'}, u'_{t'+1}, \ldots, u'_k) \in [m_1] \times \cdots \times [m_k]$.

  We first prove i).
  Assume for contradiction that $g_{\xi}(J_{\lambda}^+) \cap g_{\xi'}(J_{\lambda'}^{+}) \neq \emptyset$.
  Note that
  \begin{align*}
    g_{\xi}(J_{\lambda}^{+}) &\subset [e_k^{u_{k}-1}(x_{e_k}), e_k^{u_{k}}(x_{e_k})], \\
    g_{\xi'}(J_{\lambda'}^{+}) & \subset [e_k^{u'_{k}-1}(x_{e_k}), e_k^{u'_{k}}(x_{e_k})].
  \end{align*}
  If $u_k \neq u'_{k}$, then either of the following holds:
  \begin{enumerate}[$(1)$]
    \item $g_{\xi}(J_{\lambda}^+) \cap g_{\xi'}(J_{\lambda'}^{+}) = \{ e_k^{u_k}(x_{e_k}) \}$,
    \item $g_{\xi}(J_{\lambda}^+) \cap g_{\xi'}(J_{\lambda'}^{+}) = \{ e_k^{u'_k}(x_{e_k}) \}$.
  \end{enumerate}
  If (1) holds, then $e_k^{u_k}(x_{e_k}) \in g_{\xi}(J_{\lambda}^+)$, and hence $x_{e_k} \in e_{k-1}^{u_{k-1}} \cdots e_1^{u_1}\big( [x_{e_t}, e_{t}^{-1}(x_{e_t})] \big)$.
  This implies that $u_j = m_j$ for every $j \in [k-1]$ and $t = k$.
  This contradicts $\lambda \in \Lambda_k$.
  In the same way, (2) is also contradicts to $\lambda' \in \Lambda_k$.
  Hence we obtain $u_k = u'_k$.
  Taking this procedure inductively, we obtain $u_j = u'_j$ for every $1 \leq j \leq k$.
  Hence we have $[x_{e_t}, e_t^{-1}(x_{e_t})] \cap [x_{e_{t'}}, e_{t'}^{-1}(x_{e_{t'}})] \neq \emptyset$, which implies $t = t'$.
  This completes the proof of i).

  For ii), we assume that $\Int(g_{\xi}(J_{\lambda}^-)) \cap \Int(g_{\xi'}(J_{\lambda'}^-)) \neq \emptyset$.
  Since
  \begin{align*}
    &\Int(g_{\xi}(J_{\lambda}^{-})) \subset (e_k^{u_{k}-1}(x_{e_k}), e_k^{u_{k}}(x_{e_k})) \text{ and} \\
    &\Int(g_{\xi'}(J_{\lambda'}^{-})) \subset (e_k^{u'_{k}-1}(x_{e_k}), e_k^{u'_{k}}(x_{e_k})),
  \end{align*}
  we have $u_k = u'_{k'}$.
  Taking this procedure inductively, we have $u_j = u'_j$ for $1 \leq j \leq k$ except for $j = t, t'$, and we obtain $u_t = m_t$ and $u'_{t'} = m_{t'}$.
  This implies that
  \[
    (e_t^{u_t -1}(x_{e_t}), e_t^{u_t}(x_{e_t})) \cap (e_{t'}^{u'_{t'} -1}(x_{e_{t'}}), e_{t'}^{u'_{t'}}(x_{e_{t'}})) \neq \emptyset.
  \]
  Hence we obtain $t = t'$.
  This completes the proof of ii).

  For iii), we assume that $g_{\xi}(J_{\lambda}^+) \cap g_{\xi'}(J_{\lambda'}^{-}) \neq \emptyset$.
  By the arguments same as above, we have $u_j = u'_{j}$ for every $1 \leq j \leq k$ except for $j = t'$ and $u_{t'} = m_{t'}$.
  However, this contradicts to $\lambda' \in \Lambda_{t'}$.

  For iv), we assume that $g_{\xi}(J_{\lambda}^{+}) \cap e_k^{i_k} \cdots e_1^{i_1}(J_{h_l}^{\pm}) \neq \emptyset$.
  By the arguments similar to the above, we have $u_j = i_j$ for every $1 \leq j \leq k$ and $J_{\lambda^+} \cap J_{h_l}^{\pm} \neq \emptyset$, which is not the case.

  The proofs of v) and vi) are the same as one of iv).
  The proof of vii) is the same as one of i).
\end{proof}

\begin{lemma}\label{lem:intersections_2}
  Let $2 \leq t \leq k, \lambda \in \Lambda_t$ and $\xi \in \Xi_t$.
  Let $l \in [2n]$ and $(i_1, \cdots i_k), (j_1, \cdots, j_k) \in [m_1]\times \cdots \times [m_k]$.
  Then, the following hold;
  \begin{enumerate}
    \item $e_{k}^{j_k} \cdots e_1^{j_1} (x_{e_1})$ is not contained in $g_{\xi}(J_{\lambda}^{\pm})$,
    \item $e_{k}^{j_k} \cdots e_1^{j_1} (x_{e_1})$ is not contained in $e_k^{i_k} \cdots e_1^{i_1}(J_{h_{l}}^{\pm})$.
  \end{enumerate}
\end{lemma}

We omit the proof since it is similar to one of Lemma \ref{lem:intersections}

\begin{proposition}\label{prop:ping_pong_fin_ind}
  There exist attracting domains $D(\rho(s))$ for $s \in S_1 \cup S_2 \cup S_1^{-1} \cup S_2^{-1}$ such that the elements
  \[
    \{ D(\rho(s)) \mid s \in S_1 \cup S_2 \cup S_1^{-1} \cup S_2^{-1} \} \cup \{ e_k^{j_k} \cdots e_1^{j_1}(x_{e_1}) \mid (j_1, \ldots j_k) \in [m_1] \times \cdots \times [m_k] \}
  \]
  are pairwise disjoint.
  In particular, the elements of $\rho(S_1 \cup S_2)$ have ping-pong dynamics.
\end{proposition}

\begin{proof}
  By the definitions of $f_{\lambda}, g_{\xi}, J_{\lambda}^{\pm}$ and $J_{h_l}^{\pm}$, we have
  \begin{align*}
    &{}^{g_{\xi}}f_{\lambda}(S^1 \setminus g_{\xi}(J_{\lambda}^-)) \subset g_{\xi}(J_{\lambda}^+), \\
    &{}^{e_k^{i_k}\cdots e_1^{i_1}} h_l(S^1\setminus e_k^{i_k}\cdots e_1^{i_1} (J_{h_l}^-)) \subset e_k^{i_k}\cdots e_1^{i_1}(J_{h_l}^+).
  \end{align*}

  For a closed interval $I = [a,b] \subset S^1$ and a sufficiently small $\epsilon > 0$, we set $I_{\epsilon} = [a-\epsilon, b+\epsilon]$.
  By Lemmas \ref{lem:intersections} and \ref{lem:intersections_2}, there exists $\epsilon > 0$ such that the elements of the union
  \begin{align*}
    &\{ (g_{\xi}(J_{\lambda}^{+}))_{\epsilon} \mid 2 \leq t \leq k, \lambda \in \Lambda_t, \xi \in \Xi_t \} \\
    &\cup \{ \mathrm{Cl}(({}^{g_{\xi}}f_{\lambda})^{-1}(S^1 \setminus (g_{\xi}(J_{\lambda}^+))_{\epsilon})) \mid 2 \leq t \leq k, \lambda \in \Lambda_t, \xi \in \Xi_t \} \\
    &\cup \{ e_k^{i_k}\cdots e_1^{i_1}(J_{h_l}^{\pm}) \mid (i_1, \ldots, i_k) \in [m_1]\times \cdots \times [m_k], l \in [2n] \} \\
    &\cup \{ e_k^{i_k} \cdots e_1^{i_1}(x_{e_1}) \mid (i_1, \ldots, i_k) \in [m_1]\times \cdots \times [m_k] \}
  \end{align*}
  are pairwise disjoint.
  Here $\mathrm{Cl}(X)$ denotes the closure of $X$.
  Setting
  \begin{align*}
    &D\big(\rho({}^{g_{\xi}}f_{\lambda})\big) = (g_{\xi}(J_{\lambda}^{+}))_{\epsilon}, \\
    &D\big(\rho(({}^{g_{\xi}}f_{\lambda})^{-1})\big) = \mathrm{Cl}(({}^{g_{\xi}}f_{\lambda})^{-1}(S^1 \setminus (g_{\xi}(J_{\lambda}^+))_{\epsilon})), \\
    &D\big(\rho({}^{e_{k}^{i_k}\cdots e_{1}^{i_1}}h_l)\big) = e_k^{i_k}\cdots e_1^{i_1}(J_{h_l}^{+}), \\
    &D\big(\rho(({}^{e_{k}^{i_k}\cdots e_{1}^{i_1}}h_l)^{-1})\big) = e_k^{i_k}\cdots e_1^{i_1}(J_{h_l}^{-}),
  \end{align*}
  we obtain the proposition.
\end{proof}

\begin{corollary}
  The group $F = \langle S_1 \cup S_2 \rangle$ is a free subgroup with $[G:F] = m_1\cdots m_k$.
\end{corollary}



\subsection{Isolation of $c$}
By Proposition \ref{prop:co_from_rho}, we take a circular order $c$ whose dynamical realization with basepoint $x_{e_1}$ coincides with $\rho$.
In this subsection, we prove the following.
\begin{proposition}\label{prop:isolated}
  The circular order $c$ is isolated.
\end{proposition}

The following is a slight generalization of \cite[Lemma 4.2]{MR3887426}.

\begin{lemma}\label{lem:data_det_cyclic_order}
  Let $\GG$ be a group and $F_{r}$ a rank $r$ free subgroup of $\GG$ of finite index.
  Set $m = [\GG:F_r]$.
  Let $\{ \gamma_1 = 1, \gamma_2, \ldots, \gamma_m \}$ be a complete set of representatives of $F_r \backslash \GG$.
  Let $a_1, \ldots, a_r$ be a free generating set of $F_r$.
  Let $\sigma \colon \GG \to \Homeo_+(S^1)$ be a homomorphism such that the elements $\sigma(a_1), \ldots, \sigma(a_r)$ have ping-pong dynamics.
  Let $x_0 \in S^1$.
  Assume that the elements of
  \[
    \{ D(\sigma(a_i)) \mid i \in [r] \} \cup \{ D(\sigma(a_i)^{-1}) \mid i \in [r] \} \cup \{ \sigma(\gamma_j)(x_0) \mid j \in [m] \}
  \]
  are pairwise disjoint.
  Then, the action $\sigma$ is free at $x_0$ and the cyclic order of its orbit is completely determined by the following data:
  \begin{enumerate}[$(1)$]
    \item the cyclic order of the sets $D_l(\rho(s))$ and $\{ \sigma(\gamma_j)(x_0) \}$ for $s \in \{ a_1^{\pm 1}, \ldots, a_r^{\pm 1} \}$, and
    \item the collection of containment relations
    \begin{align*}
      &\sigma(s)(D_{l'}(t)) \subset D_{l}(s), \\
      &\sigma(s)(\sigma(\gamma_j)(x_0)) \in D_{l}(s).
    \end{align*}
  \end{enumerate}
\end{lemma}
We omit the proof since it is essentially the same as that of \cite[Lemma 4.2]{MR3887426}.

\begin{lemma}\label{lem:neighborhood_construction}
  There exists a neighborhood $U$ of $\rho$ in $\Hom(G, \Homeo_+(S^1))$ such that for every $\sigma \in U$, the following hold;
  \begin{enumerate}[$(1)$]
    \item $\sigma$ acts on $x_{e_1}$ freely,
    \item the circular order of $\{ \sigma(g)(x_{e_1}) \}_{g \in G}$ coincides with that of $\{ \rho(g)(x_{e_1}) \}_{g \in G}$.
  \end{enumerate}
\end{lemma}

\begin{proof}
  For $A \subset S^1$ and for $\epsilon >0$, let $N(A, \epsilon)$ be the open $\epsilon$-neighborhood of $A$.
  We take attracting domains $D(\rho(s))$ as in Proposition \ref{prop:ping_pong_fin_ind}, where $s \in S_1\cup S_2 \cup S_1^{-1} \cup S_2^{-1}$.
  Then there exists $\epsilon > 0$ such that the elements of the union
  \begin{align*}
    &\{ N(D(\rho(s)), \epsilon) \mid s \in S_1 \cup S_2 \cup S_1^{-1} \cup S_2^{-1} \} \\
    &\cup \{ N(e_k^{i_k} \cdots e_1^{i_1}(x_{e_1}), \epsilon) \mid (i_1, \ldots, i_k) \in [m_1]\times \cdots \times [m_k] \}
  \end{align*}
  are pairwise disjoint.
  We set
  \[
    S = S_1 \cup S_2 \cup S_1^{-1} \cup S_2^{-1} \cup \{ e_{k}^{i_k}\cdots e_{1}^{i_1} \mid (i_1, \cdots, i_k) \in [m_1] \times \cdots \times [m_k] \}
  \]
  and
  \[
    U = \left\{ \sigma \in \Hom(G, \Homeo_+(S^1)) \, \middle| \, \sup_{x \in S^1} \sup_{g \in S}  d(\rho(g)(x), \sigma(g)(x)) < \epsilon \right\},
  \]
  where $d$ denotes the standard metric on $S^1$.
  Obviously, this $U$ is an open neighborhood of $\rho$.

  Take $\sigma \in U$.
  Then, by the definition of $\epsilon$, we can take attracting domains of $\sigma(s)$ for $s \in S_1 \cup S_2 \cup S_1^{-1} \cup S_2^{-1}$ satisfying the assumption of Lemma \ref{lem:data_det_cyclic_order}.
  Moreover, the data (1) and (2) in Lemma \ref{lem:data_det_cyclic_order} for this $\sigma$ coincide with those for $\rho$.
  Hence, Lemma \ref{lem:data_det_cyclic_order} implies that $\sigma$ acts on $x_{e_1}$ freely and that the circular order of $\{ \sigma(g)(x_{e_1}) \}_{g \in G}$ coincides with one of $\{ \rho(g)(x_{e_1}) \}_{g \in G}$.
\end{proof}

\begin{proof}[Proof of Proposition \ref{prop:isolated}]
  By Proposition \ref{prop:co_from_rho}, the action $\rho$ is a dynamical realization of a circular order $c$ based at $x_{e_1}$.
  We take an open neighborhood $U$ of $\rho$ as in Lemma \ref{lem:neighborhood_construction}.
  Then, Theorem \ref{thm:dyn_real_cont} implies that there exists a neighborhood $V$ of $c$ in $CO(G)$ such that each order in $V$ has a dynamical realization based at $x_{e_1}$ contained in $U$.
  Hence, by (2) of Lemma \ref{lem:data_det_cyclic_order}, every order in $V$ coincides with $c$, which implies the isolation of $c$.
\end{proof}

\subsection{Infinitely many isolated orders on $G$}\label{subsec:inf_many_iso_circ_ord}
In this subsection, we construct infinitely many isolated circular orders by lifting the action $\rho$ to actions on finite covering spaces of $S^1$.

For an integer $d \geq 1$, let $\tau_d \colon S^1 \to S^1$ be the $(1/d)$-rotation and set
\[
  \Homeo_{+}^{(d)}(S^1) = \{ f \in \Homeo_+(S^1) \mid f \circ \tau_d = \tau_d \circ f \}.
\]
Note that the $d$-fold covering $p_d \colon S^1 \to S^1$ induces the exact sequence
\[
  0 \to \ZZ_d \to \Homeo_+^{(d)}(S^1) \xrightarrow{\pi} \Homeo_+(S^1) \to 1.
\]

For an integer $m \geq 2$, the homomorphism $\phi \colon \ZZ_m \to \Homeo_+(S^1)$ defined by $\phi(1) = \tau_m$ lifts to a homomorphism $\tphi \colon \ZZ_m \to \Homeo_+^{(d)}(S^1)$ if and only if there exists $i \in \ZZ_{\geq 0}$ such that $d$ divides $mi+1$.
Note that if $d$ divides $mi + 1$, then the lift $\tphi$ is given by setting $\tphi(1)$ as the lift of $\phi(1)$ with $\rot(\tphi(1)) = (mi+1)/dm$.
In particular, the homomorphism $\rho \colon G \to \Homeo_+(S^1)$ lifts to a homomorphism $\trho \colon G \to \Homeo_+^{(d)}(S^1)$ if and only if for every $l \in [k]$, there exists $i_l \in \ZZ_{\geq 0}$ such that $d$ divides $m_l i_l + 1$.

We take a specific lift $\rho^{(d)} \colon G \to \Homeo_+^{(d)}(S^1)$ as follows.
Assume that for every $l \in [k]$, there exists $i_l \in \ZZ_{\geq 0}$ such that $d$ divides $m_l i_l + 1$.
For every $i \in [2n]$, let $\rho^{(d)}(h_i)$ be the lift of $\rho(h_i)$ with $\rot(\rho^{(d)}(h_i)) = 0$.
For every $l \in [k]$, let $\rho^{(d)}(e_l)$ be the lift of $\rho(e_l)$ with $\rot(\rho^{(d)}(e_l)) = (m_l i_l + 1)/dm_l$.
This provides a well-defined lift $\rho^{(d)} \colon G \to \Homeo_+^{(d)}(S^1)$ of $\rho$.

Let $\widetilde{x_{e_1}}$ be a point of $p_d^{-1}(x_b)$.
Let $I_1, \cdots, I_d$ be the connected components of $p_d^{-1}(I_o)$, where $I_1$ is the interval containing $\widetilde{x_{e_1}}$ and they are aligned in this order with respect to the orientation of $S^1$.

\begin{lemma}\label{lem:transitive_I_1_covering}
  There exist infinitely many $a \in \ZZ_{\geq 0}$ such that $d = m_1 \cdots m_k \cdot a + 1$ satisfies the following;
  \begin{enumerate}
    \item for every $l \in [k]$, there exists $i_l \in \ZZ_{\geq 0}$ such that $d$ divides $m_l i_l + 1$, and
    \item for every $i \in [d]$, there exists a unique $j \in [d]$ such that $I_i = \rho^{(d)}(\alpha^j)(I_1)$, where $\alpha = e_1 \cdots e_k [h_1, h_2] \cdots [h_{2n-1}, h_{2n}]$.
  \end{enumerate}
\end{lemma}

In the proof of Lemma \ref{lem:transitive_I_1_covering}, we use the following lemma.
\begin{lemma}\label{lem:infinitely_many_rel_prime}
  Let $p_0, p_1$ be elements of $\ZZ_{\geq 1}$ and $N \in \ZZ_{\geq 0}$.
  Assume that $Np_0 \neq p_1$.
  Then, there exist infinitely many $a \in \ZZ_{\geq 0}$ such that $p_0a+1$ and $p_1a+N$ are relatively prime.
\end{lemma}

\begin{proof}
  If $p_0 > p_1$, then for sufficiently large $a$, we have $p_0 a + 1 > p_1 a + N$.
  Note that there exist infinitely many $a \in \ZZ_{\geq 0}$ such that $p_0 a + 1$ is prime by Dirichlet's theorem on arithmetic progressions.
  Hence there exist infinitely many $a \in \ZZ_{\geq 0}$ such that $p_0a+1$ and $p_1a+N$ are relatively prime.

  Assume that $p_0 \leq p_1$.
  We set $p_1 = p_0 q + r$, where $q \in \ZZ_{\geq 1}$ and $0 \leq r < p_0$.
  Then $\mathrm{gcd}(p_0 a + 1, p_1 a + N) = \mathrm{gcd}(p_0 a + 1, r a + N-q)$.
  Note that $(q,r) \neq (N,0)$ by the assumption on $p_0, p_1$ and $N$.
  Hence, by the argument similar to the case of $p_0 > p_1$, we obtain infinitely many $a \in \ZZ_{\geq 0}$ such that $p_0a+1$ and $p_1a+N$ are relatively prime.
\end{proof}

\begin{proof}[Proof of Lemma \ref{lem:transitive_I_1_covering}]
  First note that for every $a \in \ZZ_{\geq 0}$, item i) holds.

  Let $a$ be an element of $\ZZ_{\geq 0}$ and set $d = m_1 \cdots m_k \cdot a + 1$. Let $\rho^{(d)} \colon G \to \Homeo_+^{(d)}(S^1)$ be the lift defined above.
  Recall that $\alpha = e_1 \cdots e_k [h_1, h_2]\cdots [h_{2n-1}, h_{2n}]$ is a stabilizer of $I_o$.
  By tracking the orbit of a point of $\partial I_1$, It is verified that
  \[
    \rot(\rho^{(d)}(\alpha)) = \left( 2n-1+k + \displaystyle\sum_{l = 1}^{k}i_l \middle) \right/ d.
  \]
  Hence $\rho^{(d)}(\alpha)(I_i) = I_j$ if and only if
  \[
    j-i = 2n-1+k+\sum_{l = 1}^{k}i_l \ \ \text{ mod } d.
  \]
  Therefore, it suffices to prove that there exist infinitely many such $d$ such that $d$ and $2n-1+k+\sum_{l = 1}^{k}i_l$ are relatively prime.

  Set $N = k+2n-1, p_0 = m_1 \cdots m_k$ and $p_1 = \sum_{l= 1}^{k} m_1 \cdots \widehat{m_l} \cdots m_k$, where the hat means that the term under the hat is omitted.
  Here we regard $p_1$ as $1$ when $k = 1$.
  Then we have
  \begin{align}\label{align:rel_prime_candidate}
    &p_0 a + 1 = d, \nonumber \\
    &p_1 a + N = 2n-1+k+\sum_{l = 1}^{k}i_l.
  \end{align}
  If $Np_0 \neq p_1$, then applying Lemma \ref{lem:infinitely_many_rel_prime} to \eqref{align:rel_prime_candidate}, we obtain the lemma.
  Hence it suffices to show that $Np_0 \neq p_1$.

  If $k = 1$, then $Np_0 = 2nm_1$ and $p_1 = 1$. Hence $Np_0 \neq p_1$.

  If $k = 2$, then
  \[
    Np_0 - p_1 = (2n+1)m_1m_2 - (m_1 + m_2) = 2nm_1m_2 + (m_1-1)(m_2-1)-1.
  \]
  Since $2nm_1m_2 \geq 0$ and $(m_1-1)(m_2-1) \geq 1$, the equality $Np_0 = p_1$ holds if and only if $n = 0$ and $m_1 = m_2 = 2$.
  Hence, by the assumption on $G$ (that is, the group $G$ is not isomorphic to $\ZZ_2 \ast \ZZ_2$), we have $Np_0 \neq p_1$.

  If $k \geq 3$, then
  \begin{align*}
    Np_0 - p_1 &\geq (k-1)p_0 - p_1 = (k-1)m_1 \cdots m_k - \sum_{l=1}^{k}m_1 \cdots \widehat{m_l} \cdots m_k \\
    &\geq (k-1)m_1 \cdots m_k - \frac{k}{2}m_1 \cdots m_k = \frac{k-2}{2}m_1 \cdots m_k > 0.
  \end{align*}
  Hence we obtain $Np_0 \neq p_1$.
  This completes the proof.
\end{proof}

The following proposition is an analogue of Proposition \ref{prop:co_from_rho}.
\begin{proposition}\label{prop:co_from_trho}
  Let $a$ be an element of $\ZZ_{\geq 0}$ satisfying items i) and ii) in Lemma \ref{lem:transitive_I_1_covering}.
  Then the lift $\rho^{(d)} \colon G \to \Homeo_+^{(d)}(S^1)$ of $\rho$ is a dynamical realization of a circular order $c^{(d)} \in CO(G)$ with basepoint $\widetilde{x_{e_1}}$, and the linear part is the subgroup of $G$ generated by $\alpha^d$.
\end{proposition}

\begin{proof}
  By ii) of Lemma \ref{lem:transitive_I_1_covering}, the stabilizer of $I_1$ is the cyclic subgroup of $G$ generated by $\alpha^d$.
  Note that the inverse image $p_d^{-1}(L_{\rho})$ of the exceptional minimal set $L_{\rho}$ of $\rho$ gives rise to an exceptional minimal set of the lift $\rho^{(d)}$.

  To use Theorem \ref{thm:dyn_real_criterion}, we prove that for every connected component $J$ of $S^1 \setminus p_d^{-1}(L_{\rho})$, there exists $g \in G$ such that $I_1 = \rho^{(d)}(g)(J)$.
  Let $J$ be a connected component of $S^1 \setminus p_d^{-1}(L_{\rho})$.
  Then $p_d(J)$ is a connected component of $S^1 \setminus L_{\rho}$.
  Hence, by Lemma \ref{lem:transitive_I_o}, there exists $g' \in G$ such that $I_o = \rho(g')(p_d(J))$, which implies that there exists $i \in [d]$ such that $I_i = \rho^{(d)}(g')(J)$.
  By Lemma \ref{lem:infinitely_many_rel_prime}, there exists $j \in [d]$ such that $I_i = \rho^{(d)}(\alpha^j)(I_1)$.
  Hence we have $I_1 = \rho^{(d)}(\alpha^{-j})(I_i) = \rho^{(d)}(\alpha^{-j}g')(J)$.

  Applying Theorems \ref{thm:dyn_real_criterion} and \ref{thm:linear_part_dyn_description} to $\rho^{(d)}$, we obtain the proposition.
\end{proof}

Take attracting domains $D(\rho(s))$ for $s \in S_1 \cup S_2 \cup S_1^{-1} \cup S_2^{-1}$ as in Proposition \ref{prop:ping_pong_fin_ind}.
Let $d$ be an integer satisfying items i) and ii) in Lemma \ref{lem:transitive_I_1_covering} and set $D(\rho^{(d)}(s)) = p_d^{-1}(D(\rho(s)))$.
Since $p_d \colon S^1 \to S^1$ is the $d$-fold covering and $\rho^{(d)}$ is a lift of $\rho$, we have the following.
\begin{lemma}
  The closed sets $D(\rho^{(d)}(s)) = p_d^{-1}(D(\rho(s)))$ are attracting domains for $s \in S_1 \cup S_2 \cup S_1^{-1} \cup S_2^{-1}$ such that the elements
  \[
    \{ D(\rho^{(d)}(s)) \mid s \in S_1 \cup S_2 \cup S_1^{-1} \cup S_2^{-1} \} \cup \{ \rho^{(d)}(e_k^{j_k} \cdots e_1^{j_1})(\widetilde{x_{e_1}}) \mid (j_1, \ldots j_k) \in [m_1] \times \cdots \times [m_k] \}
  \]
  are pairwise disjoint.
\end{lemma}

The following is an analogue of Lemma \ref{lem:neighborhood_construction}.
\begin{lemma}
  Let $a$ be an element of $\ZZ_{\geq 0}$ satisfying items i) and ii) in Lemma \ref{lem:transitive_I_1_covering}.
  Then there exists a neighborhood $U$ of $\rho^{(d)}$ in $\Hom(G, \Homeo_+(S^1))$ such that for every $\sigma \in U$, the following hold;
  \begin{enumerate}[$(1)$]
    \item $\sigma$ acts on $\widetilde{x_{e_1}}$ freely,
    \item the circular order of $\{ \sigma(g)(\widetilde{x_{e_1}}) \}_{g \in G}$ coincides with one of $\{ \rho^{(d)}(g)(\widetilde{x_{e_1}}) \}_{g \in G}$.
  \end{enumerate}
\end{lemma}

We omit the proof since it is essentially the same as that of Lemma \ref{lem:neighborhood_construction}.
As a corollary, we obtain the following.
\begin{corollary}\label{cor:isolation_cd}
  Let $a$ be an element of $\ZZ_{\geq 0}$ satisfying items i) and ii) in Lemma \ref{lem:transitive_I_1_covering}.
  Then the circular order $c^{(d)}$ is isolated.
\end{corollary}

We now prove that each isolated circular order $c^{(d)}$ is not an automorphic image of the others, which implies Theorem \ref{thm:main}.
To prove this, we need the following.

\begin{lemma}\label{lem:linear_part_auto_image}
  Let $\GG$ be a group and $c$ a circular order on $\GG$.
  Let $\GG_0$ be the linear part of $c$.
  Then for every automorphism $\phi$ of $\GG$, the linear part of the induced order $c_{\phi}$ is $\phi^{-1}(\GG_0)$.
\end{lemma}

\begin{proof}
  We fix an enumeration $\GG = \{ \gg_i \}_{i \geq 0}$ with $\gg_0 = 1$.
  Let $\sigma$ be the dynamical realization of $c$ based at $x_0$ with respect to the enumeration above.
  Let $\sigma_{\phi}$ be the dynamical realization of $c_{\phi}$ based at $x_0$ with respect to the enumeration $\GG = \{ \phi(\gg_i) \}_{i \geq 0}$.
  By the definition of the dynamical realization, we have $\sigma_{\phi} = \sigma \circ \phi$.
  Hence the exceptional minimal set $K$ of $\sigma_{\phi}$ coincides with one of $\sigma$.

  Let $I$ be the connected component of $S^1 \setminus K$ containing $x_0$.
  By Theorem \ref{thm:linear_part_dyn_description}, the linear part of $c$ (resp. of $c_{\phi}$) is the stabilizer $\mathrm{Stab}_{\sigma}(I)$ (resp. $\mathrm{Stab}_{\sigma_{\phi}}(I)$) of $I$ with respect to $\sigma$ (resp. $\sigma_{\phi}$).
  Since $\mathrm{Stab}_{\sigma_{\phi}}(I) = \phi^{-1}(\mathrm{Stab}_{\sigma}(I)) = \phi^{-1}(\GG_0)$, we obtain the lemma.
\end{proof}

\begin{proof}[Proof of Theorem \ref{thm:main}]
  By Lemma \ref{lem:infinitely_many_rel_prime} and Corollary \ref{cor:isolation_cd}, it suffices to prove that each isolated circular order $c^{(d)}$ is not an automorphic image of the others.

  Let $d$ and $d'$ be integers satisfying items i) and ii) of Lemma \ref{lem:infinitely_many_rel_prime}.
  Then the linear part of $c^{(d)}$ (resp. $c^{(d')}$) is the cyclic subgroup of $G$ generated by $\alpha^d$ (resp. $\alpha^{d'}$) by Proposition \ref{prop:co_from_trho}.
  Assume that there exists an automorphism $\phi$ of $G$ such that $c^{(d)} = (c^{(d')})_{\phi}$.
  Then Lemma \ref{lem:linear_part_auto_image} implies that the cyclic subgroup $\langle \alpha^d \rangle$, which is the linear part of $c^{(d)}$, coincides with $\phi^{-1} (\langle \alpha^{d'} \rangle)$.
  Hence we have either $\phi(\alpha)^d = \alpha^{d'}$ or $\phi(\alpha)^d = \alpha^{-d'}$.

  Note that $\alpha$ is a primitive element of $G$.
  Hence so is $\phi(\alpha)$.
  Since $G$ can be embedded into $\Isom_+(\HH^2)$ as a discrete subgroup, every element of $G$ of infinite order is a multiple of a unique primitive element.
  Hence, $\phi(\alpha)^d = \alpha^{\pm d'}$ implies that $\phi(\alpha) = \alpha^{\pm 1}$ and $d = d'$.
  This completes the proof.
\end{proof}

\section{Isolated left orders}
In this section, we prove Theorem \ref{thm:main_left}, which states that the group
\[
  \hG = \langle e_1, \ldots, e_k, h_1, \ldots, h_{2n}, z \mid e_1^{m_1} = \cdots e_k^{m_k} = z, [h_1, z] = \cdots = [h_{2n}, z] = 1 \rangle
\]
admits countably many isolated left orders which are not the automorphic images of the others.

\subsection{Isolated left orders on a central $\ZZ$-extension}
In this subsection, we recall Mann--Rivas's criterion (\cite{MR3887426}) for certain left orders on a central $\ZZ$-extension to be isolated in terms of circular orders.

Let $<$ be a left order on a group $\GG$.
A subgroup $\Lambda$ of $\GG$ is said to be \emph{cofinal for} $<$ if for every $\gg \in \GG$, there exist $\lambda_1, \lambda_2 \in \Lambda$ such that $\lambda_1 < \gg < \lambda_2$.
Let $LO_{\Lambda}(\GG)$ be the set of left orders on $\GG$ for which the subgroup $\Lambda$ is cofinal.

\begin{theorem}[{\cite[Proposition 5.1]{MR3887426}}]\label{thm:circ_order_lift_Z_ext}
  Let $\GG$ be a group and $0 \to \ZZ \to \hGG \to \GG \to 1$ a central $\ZZ$-extension of $\GG$.
  Then, there is a continuous map $\pi_{\hGG}^* \colon LO_{\ZZ}(\hGG)\to CO(\GG)$.
  Moreover, each circular order on $\GG$ is in the image of one such map $\pi_{\hGG}^*$.
\end{theorem}

The map $\pi_{\hGG}^* \colon LO_{\ZZ}(\hGG) \to CO(\GG)$ is defined as follows.
Let $<$ be an element of $LO_{\ZZ}(\hGG)$.
Let $z$ be the generator of $\ZZ \subset \hGG$ satisfying $1 < z$.
Since $\ZZ$ is cofinal for $<$, for every $\gg \in \GG$, there exists a unique lift $\hgg \in \hGG$ of $\gg$ such that $1 \leq \hgg \leq z$.
For every distinct elements $\gg_1, \gg_2, \gg_3 \in \GG$, let $\sigma$ be the permutation such that $1 \leq \hgg_{\sigma(1)} < \hgg_{\sigma(2)} < \hgg_{\sigma(3)} < z$.
The map $\pi_{\hGG}^*$ is defined by $\pi_{\hGG}^*(<)(\gg_1, \gg_2, \gg_3) = \mathrm{sign}(\sigma)$.

For a countable group $\GG$, the central extension $\hGG$ in the last statement of Theorem \ref{thm:circ_order_lift_Z_ext} is given as follows.
Let $c$ be a circular order on $\GG$.
Let $\rho$ be a dynamical realization of $c$ based at $x$.
Let $T \colon \RR \to \RR$ be the translation by one and $\tHomeo_+(S^1)$ the group of homeomorphisms of $\RR$ which commute with $T$.
Note that the group $\tHomeo_+(S^1)$ gives rise to a central $\ZZ$-extension of $\Homeo_+(S^1)$.
Then the group $\hGG$ is given as the pullback of $\tHomeo_+(S^1)$.

Take a lift $\hat{x} \in \RR$ of $x$.
Then $\hat{x}$ has a free orbit under $\hGG\subset \Homeo(\RR)$, which induces
a left order $<$ on $\hGG$ for which $\ZZ$ is cofinal.
This left order $<$ satisfies $\pi_{\hGG}^*(<) = c$.
(See \cite{MR3887426} for details.)

The isolation of certain left orders on $\hGG$ can be detected by the isolation of the induced circular orders on $\GG$ as follows.
\begin{theorem}[{\cite[Proposition 5.4]{MR3887426}}]\label{thm:criterion_isol_left_order}
  Assume that $\GG$ is finitely generated and $c$ is an isolated circular order on $G$.
  If $0 \to \ZZ \to \hat{\GG} \to \GG \to 1$ is a central extension and $< \in LO_{\ZZ}(\hGG)$ is a left order such that $\pi_{\hGG}^*(<) = c$, then $<$ is isolated in $LO(\hGG)$.
\end{theorem}

Let $\hphi \colon \hGG \to \hGG$ be an automorphism such that $\hphi(\ZZ) = \ZZ$.
Then $\hphi$ descends to an automorphism $\phi \colon \GG \to \GG$ and forms the following commutative diagram:
\[
\xymatrix{
0 \ar[r] & \ZZ \ar[r] \ar[d] & \hGG \ar[r] \ar[d]^-{\hphi} & \GG \ar[r] \ar[d]^-{\phi} & 1 \\
0 \ar[r] & \ZZ \ar[r] & \hGG \ar[r] & \GG \ar[r] & 1.
}
\]


\begin{lemma}\label{lem:induced_order_compatible}
  The automorphisms $\hphi$ and $\phi$ induce the following commutative diagram
  \[
  \xymatrix{
  LO_{\ZZ}(\hGG) \ar[r]^-{\hphi^*} \ar[d]^-{\pi_{\hGG}^*} & LO_{\ZZ}(\hGG) \ar[d]^-{\pi_{\hGG}^*} \\
  CO(\GG) \ar[r]^-{\phi^*} & CO(\GG).
  }
  \]
  Here $\hphi^*$ (resp. $\phi^*$) is given by $\hphi^*(c) = c_{\hphi}$ (resp. $\phi^*(<) = <_{\phi}$).
\end{lemma}

\begin{proof}
  First we prove that for $< \in LO_{\ZZ}(\hGG)$, the induced order $<_{\hphi}$ is also an element of $LO_{\ZZ}(\hGG)$.
  Let $\hgg$ be an element of $\hGG$.
  Since $\ZZ$ is cofinal for $<$, there exist $z^k$ and $z^l$ such that $z^k < \hphi(\hgg) < z^l$.
  By the definition of $<_{\hphi}$, we have $\hphi^{-1}(z^k) <_{\hphi} \hgg <_{\hphi} \hphi^{-1}(z^l)$.
  Since $\hphi^{-1}(z^k), \hphi^{-1}(z^k) \in \ZZ$, the subgroup $\ZZ$ is cofinal for $<_{\hphi}$.
  This implies that $\hphi \colon LO_{\ZZ}(\hGG) \to LO_{\ZZ}(\hGG)$ is well defined.

  Let $<$ be an element of $LO_{\ZZ}(\hGG)$.
  Let $z$ be the generator of $\ZZ$ satisfying $1 <_{\hphi} z$.
  Note that $\hphi(z)$ is the generator of $\ZZ$ satisfying $1 < \hphi(z)$.
  Let $\gg$ be an element of $\GG$ and $\hgg \in \hGG$ be the unique lift of $\gg$ satisfying $1 \leq_{\hphi} \hgg <_{\hphi} z$.
  Then the element $\hphi(\hgg)$ is the unique lift of $\phi(\gg)$ satisfying $1 = \hphi(1) \leq \hphi(\hgg) < \hphi(z)$.

  For distinct elements $\gg_1, \gg_2, \gg_3$ of $\GG$, let $\sigma$ be the permutation such that
  \begin{align}\label{align:sign_def}
    1 \leq_{\hphi} \hgg_{\sigma(1)} <_{\hphi} \hgg_{\sigma(2)} <_{\hphi} \hgg_{\sigma(3)} <_{\hphi} z.
  \end{align}
  This implies that
  \[
    \pi_{\hGG}^* \hphi^* (<)(\gg_1, \gg_2, \gg_3) = \pi_{\hGG}^*(<_{\hphi})(\gg_1, \gg_2, \gg_3) = \mathrm{sign}(\sigma).
  \]

  By \eqref{align:sign_def}, we obtain
  \[
    1 \leq \hphi(\hgg_{\sigma(1)}) < \hphi(\hgg_{\sigma(2)}) < \hphi(\hgg_{\sigma(3)}) < \hphi(z).
  \]
  This implies that
  \[
    \phi^* \pi_{\hGG}^*(<)(\gg_1, \gg_2, \gg_3) = \pi_{\hGG}^*(<)(\phi(\gg_1), \phi(\gg_2), \phi(\gg_3)) = \mathrm{sign}(\sigma).
  \]
  Hence we obtain $\pi_{\hGG}^* \hphi^* = \phi^* \pi_{\hGG}^*$.
\end{proof}

\subsection{Proof of Theorem \ref{thm:main_left}}
Recall that
\[
  \hG = \langle e_1, \ldots, e_k, h_1, \ldots, h_{2n}, z \mid e_1^{m_1} = \cdots = e_k^{m_k} = z, [h_1, z] = \cdots = [h_{2n}, z] = 1 \rangle.
\]
The subgroup $\langle z \rangle \cong \ZZ$ coincides with the center of $\hG$, and the quotient $\hG / \langle z \rangle$ is isomorphic to $G$.
Hence we obtain a central extension
\[
  0 \to \ZZ \to \hG \to G \to 1.
\]

\begin{lemma}\label{lem:cent_ext_determine}
  Let $a$ be an element of $\ZZ_{\geq 0}$ satisfying items i) and ii) in Lemma \ref{lem:transitive_I_1_covering} and set $d = m_1 \cdots m_k \cdot a + 1$.
  Let $\rho^{(d)} \colon G \to \Homeo_+^{(d)} \subset \Homeo_+(S^1)$ be the action defined in Subsection \ref{subsec:inf_many_iso_circ_ord}.
  Then the pullback $(\rho^{(d)})^* (\tHomeo_+(S^1))$ is isomorphic to $\hG$.
\end{lemma}

\begin{proof}
  It suffices to construct a homomorphism $\trhod \colon \hG \to \tHomeo_+(S^1)$ such that the following diagram commutes:
  \begin{align*}
    \xymatrix{
    0 \ar[r] & \ZZ \ar[r] \ar@{=}[d] & \hG \ar[r] \ar[d]^-{\trhod} & G \ar[r] \ar[d]^-{\rho^{(d)}} & 1 \\
    0 \ar[r] & \ZZ \ar[r] & \tHomeo_+(S^1) \ar[r] & \Homeo_+(S^1) \ar[r] & 1.
    }
  \end{align*}

  By the definitions of $d$ and $\rho^{(d)}$, we have
  \[
    \rot(\rho^{(d)}(e_l)) = \dfrac{m_li_l+1}{dm_l} = \dfrac{1}{m_l}
  \]
  and $\rot(\rho^{(d)}(h_i)) = 0$.
  Let $\trhod(e_l) \in \tHomeo_+(S^1)$ (resp. $\trhod(h_i) \in \tHomeo_+(S^1)$) be the lift of $\rho^{(d)}(e_l)$ (resp. $\rho^{(d)}(h_i)$) with $\trot(\trhod(e_l)) = 1/m_l$ (resp. $\trot(\trhod(h_i)) = 0$), where $\trot$ denotes the Poincar\'{e} translation number.
  Let $\trho(z) = T$, the translation by one.
  Then this $\trhod$ makes the above diagram commute.
\end{proof}

\begin{proof}[Proof of Theorem \ref{thm:main_left}]
  Let $a$ be an element of $\ZZ_{\geq 0}$ satisfying items i) and ii) in Lemma \ref{lem:transitive_I_1_covering}.
  By Proposition \ref{prop:co_from_trho}, Theorem \ref{thm:circ_order_lift_Z_ext} and Lemma \ref{lem:cent_ext_determine}, we obtain a left order $<^{(d)} \in LO_{\ZZ}(\hG)$ such that $\pi_{\hG}^*(<^{(d)}) = c^{(d)}$.
  Together with Theorem \ref{thm:criterion_isol_left_order}, the left order $<^{(d)}$ is isolated in $LO(\hG)$.

  Assume that there exists an automorphism $\hphi \colon \hG \to \hG$ such that $<^{(d)}$ coincides with $(<^{(d')})_{\hphi}$.
  Since $\hphi$ preserves the center $\langle z \rangle \cong \ZZ$ of $\hG$, the automorphism $\hphi$ descends to an automorphism $\phi \colon G \to G$.
  By Lemma \ref{lem:induced_order_compatible}, we obtain
  \[
    c^{(d)} = \pi_{\hG}^*(<^{(d)}) = \pi_{\hG}^*((<^{(d')})_{\hphi}) = (c^{(d')})_{\phi}.
  \]
  By Theorem \ref{thm:main}, we have $d = d'$.
  This completes the proof.
\end{proof}

\section*{Acknowledgments}
C. Jibiki is supported by JST SPRING, Japan Grant Number JPMJSP2106.
S. Maruyama is partially supported by JSPS KAKENHI Grant Number JP23KJ1938 and JP23K12971.

\bibliographystyle{amsalpha}
\bibliography{202403ref.bib}
\end{document}